\documentclass[a4paper,11pt]{article}
\usepackage{a4wide}
\usepackage[utf8]{inputenc}
\usepackage{latexsym,amsfonts,amsmath,amssymb,mathrsfs,url,amsthm,color}

\newtheorem{theorem}{Theorem}[section]
\newtheorem{definition}[theorem]{Definition}
\newtheorem{lemma}[theorem]{Lemma}

\newtheorem{remark}[theorem]{Remark}

\newcommand{\rd}{\, \mathrm{d}}
\newcommand{\bsc}{\boldsymbol{c}}
\newcommand{\bsh}{\boldsymbol{h}}

\newcommand{\bsx}{\boldsymbol{x}}
\newcommand{\bsy}{\boldsymbol{y}}
\newcommand{\bsz}{\boldsymbol{z}}

\newcommand{\bsgamma}{\boldsymbol{\gamma}}
\newcommand{\bszero}{\boldsymbol{0}}
\newcommand{\EE}{\mathbb{E}}
\newcommand{\RR}{\mathbb{R}}
\newcommand{\NN}{\mathbb{N}}
\newcommand{\ZZ}{\mathbb{Z}}

\newcommand{\tr}{\mathrm{tr}}
\newcommand{\wor}{\mathrm{wor}}

\allowdisplaybreaks

\title{A note on concatenation of quasi-Monte Carlo and plain Monte Carlo rules in high dimensions}
\author{Takashi Goda\thanks{School of Engineering, University of Tokyo, 7-3-1 Hongo, Bunkyo-ku, Tokyo 113-8656, Japan. ({\tt goda@frcer.t.u-tokyo.ac.jp})}}
\date{\today}

\begin{document}

\maketitle
\begin{abstract}
    In this note, we study a concatenation of quasi-Monte Carlo and plain Monte Carlo rules for high-dimensional numerical integration in weighted function spaces. In particular, we consider approximating the integral of periodic functions defined over the $s$-dimensional unit cube by using rank-1 lattice point sets only for the first $d\, (<s)$ coordinates and random points for the remaining $s-d$ coordinates. We prove that, by exploiting a decay of the weights of function spaces, almost the optimal order of the mean squared worst-case error is achieved by such a concatenated quadrature rule as long as $d$ scales at most linearly with the number of points. This result might be useful for numerical integration in extremely high dimensions, such as partial differential equations with random coefficients for which even the standard fast component-by-component algorithm is considered computationally expensive.
\end{abstract}

\section{Introduction}
We study numerical integration of functions defined over the $s$-dimensional unit cube with large $s$. For an integrable function $f: [0,1)^s\to \RR$, we write
\[ I(f):= \int_{[0,1)^s}f(\bsx)\rd \bsx. \]
For a point set $P\subset [0,1)^s$, the integral $I(f)$ is approximated by an equally-weighted quadrature rule
\[ I(f;P) := \frac{1}{|P|}\sum_{\bsx\in P}f(\bsx). \]
We call $I(f;P)$ a quasi-Monte Carlo (QMC) rule when the point set $P$ is chosen in a deterministic manner, while we call $I(f;P)$ a plain Monte Carlo (MC) rule when each point $\bsx\in P$ is generated independently from the uniform distribution over the domain $[0,1)^s$. In this paper we consider an intermediate rule between QMC and MC rules, in the sense that the projection of $P$ onto the first $d$ coordinates with some $d<s$ is given by a deterministic QMC point set, whereas the projection onto the remaining $s-d$ coordinates is simply a set of random points. That is, we \emph{concatenate} QMC and MC rules.

Our motivation behind introducing such a concatenated rule mainly comes from applications to extremely high-dimensional problems such as partial differential equations with random coefficients \cite{KSS12,KN16,Ka19,EKN20}. For such applications, an infinite Karhunen-Lo\`{e}ve expansion needs to be truncated up to some finite terms to perform computations, whose number is the problem dimension $s$. Therefore, $s$ can be very large in order to keep the truncation error small. Even so, the notion of weighted function spaces from \cite{SW98} can be successfully applied, often resulting in a dimension-independent error bound supported by the (fast) component-by-component (CBC) construction of QMC point sets \cite{SR02,Kuo03,NC06,DG21}.

Typically, for the case of product weights, the fast CBC construction of rank-1 lattice point sets, a special class of QMC point sets, requires $O(sN\log N)$ arithmetic operations with $O(N)$ memory \cite{NC06}, where $N$ denotes the number of points. When $s$ is as large as or even much larger than $N$, we may need some techniques to reduce the construction cost from the standard fast CBC algorithm. Some relevant ideas in this direction have been introduced in \cite{DKLP15,DK16}. For instance, in \cite{DKLP15}, the authors assume a fast decay of the weight parameters and introduce a fast CBC algorithm which reduces the size of the search space for coordinates with less relative importance. The search space finally contains only one possible choice after some coordinates, so that we do not need to search any more. This idea has been further explored in \cite{KPW16} as a truncation algorithm, in which the latter coordinates are fixed at an anchor $c\in [0,1)$. In this paper, instead of reducing the size of the search space or fixing at an anchor, we apply the standard fast CBC algorithm up to the first $d$ coordinates, and then the remaining $s-d$ coordinates with less relative importance are filled by random points. Hence the necessary construction cost is of $O(dN\log N)$, which may lead to a substantial cost saving if $d$ can be much smaller than $s$. 

Here we point out that the idea of concatenating quasi-Monte Carlo and plain Monte Carlo rules goes back to the work of Spanier \cite{Sp95} and then has been investigated in terms of discrepancy of concatenated, or mixed/hybrid, point sets and sequences, see \cite{Ok96,Gn09,Ni09,AH12} among others. In \cite{Ow94,Ow98}, Owen studied the variance of concatenated rules between quasi-Monte Carlo and Latin hypercube samplings. We also refer to \cite[Section~4.5]{WGH20} for some results on concatenating two randomized sequences.

The rest of this note is organized as follows. In the next section, we introduce the function space of our interest, namely, the weighted Korobov space consisting of smooth periodic functions over $[0,1)^s$. In Section~\ref{sec:rule}, we describe our concatenated rule made up of a rank-1 lattice point set and random points. As the main result of this note, we show in Section~\ref{sec:error} that almost the optimal order of the mean squared worst-case error is achieved if $d$ scales at most linearly with the number of points. We conclude this paper with comparing the result for our concatenated rule with the results for the reduced CBC algorithm and for a truncation algorithm, respectively.

\section{Weighted Korobov spaces}\label{sec:space}
Let $f: [0,1)^s\to \RR$ be periodic and given by its absolutely convergent Fourier series
\[ f(\bsx) = \sum_{\bsh\in \ZZ^s}\hat{f}(\bsh)\exp\left( 2\pi i \bsh\cdot \bsx\right),\]
where the dot product $\cdot$ denotes the usual inner product of two vectors on the Euclidean space $\RR^s$ and $\hat{f}(\bsh)$ denotes the $\bsh$-th Fourier coefficient of $f$:
\[ \hat{f}(\bsh) := \int_{[0,1)^s}f(\bsx)\exp\left( -2\pi i \bsh\cdot \bsx\right)\rd \bsx. \]
We measure the smoothness of periodic functions by a parameter $\alpha>1/2$. A non-increasing sequence of \emph{weights} $\gamma_1\geq \gamma_2\geq \cdots>0$ plays a role in moderating the relative importance of different variables and we write $\bsgamma=(\gamma_1,\gamma_2,\ldots)\in \RR_{>0}^{\NN}$. For a vector $\bsh\in \ZZ^s$, we define
\[ r_{\alpha,\bsgamma}(\bsh):=\prod_{\substack{j=1\\ h_j\neq 0}}^{s}\frac{\gamma_j}{|h_j|^{2\alpha}},\]
where the empty product is set to 1.
Then the weighted Korobov space, denoted by $H_{s,\alpha,\bsgamma}$, is a reproducing kernel Hilbert space with the reproducing kernel
\[ K_{s,\alpha,\bsgamma}(\bsx,\bsy) = 1+\sum_{\bsh\in \ZZ^s\setminus \{\bszero\}}r_{\alpha,\bsgamma}(\bsh)\exp\left( 2\pi i \bsh\cdot (\bsx-\bsy)\right). \]
and the inner product
\[ \langle f,g\rangle_{s,\alpha,\bsgamma} = \hat{f}(\bszero)\overline{\hat{g}(\bszero)}+\sum_{\bsh\in \ZZ^s\setminus \{\bszero\}}\frac{\hat{f}(\bsh)\overline{\hat{g}(\bsh)}}{r_{\alpha,\bsgamma}(\bsh)}.\]
The parameter $\alpha$ not only moderates the decay of the Fourier coefficients, but also coincides precisely with the number of available square-integrable partial mixed derivatives in each variable when it is an integer \cite[Section~5.8]{DKS13}.

The worst-case error of an equally-weighted quadrature rule with a fixed point set $P$ for the weighted Korobov space is defined by
\[ e^{\wor}(H_{s,\alpha,\bsgamma},P):= \sup_{\substack{f\in H_{s,\alpha,\bsgamma}\\ \|f\|_{s,\alpha,\bsgamma}\leq 1}}\left| I(f;P)-I(f)\right|.\]
It is well-known (see, for instance, \cite[Theorem~3.5]{DKS13}) that we have
\begin{align}
    (e^{\wor}(H_{s,\alpha,\bsgamma},P))^2 & = \int_{[0,1)^s}\int_{[0,1)^s}K_{s,\alpha,\bsgamma}(\bsx,\bsy)\rd \bsx \rd \bsy \notag \\
    & \quad - \frac{2}{|P|}\sum_{\bsx\in P}\int_{[0,1)^s}K_{s,\alpha,\bsgamma}(\bsx,\bsy)\rd \bsy+\frac{1}{|P|^2}\sum_{\bsx,\bsy\in P}K_{s,\alpha,\bsgamma}(\bsx,\bsy) \notag \\
    & = -1+\frac{1}{|P|^2}\sum_{\bsx,\bsy\in P}K_{s,\alpha,\bsgamma}(\bsx,\bsy), \label{eq:worst-case_error}
\end{align}
where the second equality comes from the equality
\[ \int_{[0,1)^s}K_{s,\alpha,\bsgamma}(\bsx,\bsy)\rd \bsy=1,\]
which holds for any $\bsx\in [0,1)^s$. As a reference value, the initial error is defined and given by
\[ e^{\wor}(H_{s,\alpha,\bsgamma},0):= \sup_{\substack{f\in H_{s,\alpha,\bsgamma}\\ \|f\|_{s,\alpha,\bsgamma}\leq 1}}\left| I(f)\right| = \left( \int_{[0,1)^s}\int_{[0,1)^s}K_{s,\alpha,\bsgamma}(\bsx,\bsy)\rd \bsx \rd \bsy \right)^{1/2}=1.\]

\section{Concatenation of QMC and MC rules}\label{sec:rule}
As stated in the introduction, we consider using a rank-1 lattice point set for the first $d$ coordinates with some $0<d<s$ and random points for the remaining $s-d$ coordinates. First, rank-1 lattice point sets are defined as follows.
\begin{definition}[Rank-1 lattice point sets]
Let $N,d$ be positive integers and $\bsz\in \{1,\ldots,N-1\}^d$. The rank-1 lattice point set determined by $N,d$ and $\bsz$ is given by
\[ P_{N,d,\bsz} = \left\{ \left( \left\{ \frac{nz_1}{N}\right\},\ldots,\left\{ \frac{nz_d}{N}\right\}\right) \mid 0\leq n\leq N-1\right\} \subset [0,1)^d, \]
where $\{x\}:=x-\lfloor x\rfloor$ denotes the fractional part of a real number $x$. The vector $\bsz$ is called a generating vector.
\end{definition}

Evidently we need to construct a good generating vector $\bsz$. The standard (fast) CBC algorithm is a greedy algorithm which successively searches for one component $z_j$ from the set $\{1,\ldots,N-1\}$ at a time while keeping previous components $z_1,\ldots,z_{j-1}$ unchanged. To be precise, the CBC algorithm using the squared worst-case error as a criterion proceeds as follows:
\begin{enumerate}
    \item Set $z_j=1$.
    \item For $j=2,3,\ldots,d$, choose 
    \[ z_{j}=\arg\min_{z\in \{1,\ldots,N-1\}}(e^{\wor}(H_{j,\alpha,\bsgamma},P_{N,j,(z_1,\ldots,z_{j-1},z)}))^2.\]
\end{enumerate}
\noindent As already pointed out, error bounds resulting from the CBC algorithm have been well-established in the literature, such as \cite{SR02,Kuo03,NC06,DG21} among many others. 

The concept of the dual lattice and its associated character property, described below, play a central role in our error analysis.
\begin{definition}[Dual lattice]
Let $N,d$ be positive integers and $\bsz\in \{1,\ldots,N-1\}^d$. The dual lattice of the rank-1 lattice point set $P_{N,d,\bsz}$ is defined by
\[ P^{\perp}_{N,d,\bsz} = \left\{ \bsh\in \ZZ^d \mid \bsh\cdot \bsz \equiv 0 \pmod N\right\} \subset \ZZ^d. \]
\end{definition}
\begin{lemma}[Character property]\label{lem:character}
Let $N,d$ be positive integers and $\bsz\in \{1,\ldots,N-1\}^d$. With the notation above, it holds for $\bsh\in \ZZ^d$ that
\[ \frac{1}{N}\sum_{\bsx\in P_{N,d,\bsz}}\exp(2\pi i \bsh\cdot \bsx)=\begin{cases} 1 & \text{if $\bsh\in P^{\perp}_{N,d,\bsz}$,} \\ 0 & \text{otherwise.}\end{cases} \]
\end{lemma}

Now our concatenated point set is given as follows:
\[ P_{N,s,d,\bsz} = \left\{ \left( \left\{ \frac{nz_1}{N}\right\},\ldots,\left\{ \frac{nz_d}{N}\right\},x_{n,d+1},\ldots,x_{n,s}\right) \mid 0\leq n\leq N-1\right\}\subset [0,1)^s,\]
where $x_{n,j},\, n=0,\ldots,N-1,\, j=d+1,\ldots,s$ are sampled independently from the uniform distribution on $[0,1)$.

\section{An error bound}\label{sec:error}
In what follows, $\EE$ denotes the expected value with respect to the probability measure, under which the random variables $x_{n,j},\, n=0,\ldots,N-1,\, j=d+1,\ldots,s$ are independently uniformly distributed on $[0,1)$. Moreover, for two integers $a,b$ with $a\leq b$, we write $a:b=\{a,a+1,\ldots,b\}$. First we prove the following lemma.
\begin{lemma}\label{lem:ms_worst-case_error}
Let $N,s,d$ be positive integers with $d<s$ and $\alpha>1/2$ be a real number. The mean squared worst-case error of an equally-weighted quadrature rule using a point set $P_{N,s,d,\bsz}$ for the weighted Korobov space $H_{s,\alpha,\bsgamma}$ is given by
\begin{align}\label{eq:ms_worst-case_error} \EE\left[(e^{\wor}(H_{s,\alpha,\bsgamma},P_{N,s,d,\bsz}))^2\right] = \frac{1}{N}\sum_{\substack{\bsh\in \ZZ^s\setminus \{\bszero\}\\ \bsh_{d+1:s} \neq\bszero}}r_{\alpha,\bsgamma}(\bsh)+\sum_{\bsh_{1:d}\in P^{\perp}_{N,d,\bsz}\setminus \{\bszero\}}r_{\alpha,\bsgamma}(\bsh_{1:d},\bszero).
\end{align}
\end{lemma}

\begin{proof}
Using \eqref{eq:worst-case_error}, the linearity of expectation and Lemma~\ref{lem:character}, we have
\begin{align*}
    & \EE\left[(e^{\wor}(H_{s,\alpha,\bsgamma},P_{N,s,d,\bsz}))^2\right] \\
    & = \EE\left[-1+\frac{1}{N^2}\sum_{\bsx,\bsy\in P_{N,s,d,\bsz}}K_{s,\alpha,\bsgamma}(\bsx,\bsy)\right] \\
    & = \EE\left[\frac{1}{N^2}\sum_{\bsx,\bsy\in P_{N,s,d,\bsz}}\sum_{\bsh\in \ZZ^s\setminus \{\bszero\}}r_{\alpha,\bsgamma}(\bsh)\exp\left( 2\pi i \bsh\cdot (\bsx-\bsy)\right)\right] \\
    & = \sum_{\bsh\in \ZZ^s\setminus \{\bszero\}}r_{\alpha,\bsgamma}(\bsh)\cdot \EE\left[\frac{1}{N^2}\sum_{\bsx,\bsy\in P_{N,s,d,\bsz}}\exp\left( 2\pi i \bsh\cdot (\bsx-\bsy)\right)\right] \\
    & = \sum_{\bsh\in \ZZ^s\setminus \{\bszero\}}r_{\alpha,\bsgamma}(\bsh) \cdot \EE\left[\frac{1}{N^2}\sum_{\bsx\in P_{N,s,d,\bsz}}\exp\left( 2\pi i \bsh\cdot \bszero \right)+\frac{1}{N^2}\sum_{\substack{\bsx,\bsy\in P_{N,s,d,\bsz}\\ \bsx\neq \bsy}}\exp\left( 2\pi i \bsh\cdot (\bsx-\bsy)\right)\right] \\
    & = \sum_{\bsh\in \ZZ^s\setminus \{\bszero\}}r_{\alpha,\bsgamma}(\bsh) \cdot \left[\frac{1}{N} +\frac{1}{N^2}\sum_{\substack{\bsx,\bsy\in P_{N,d,\bsz}\\ \bsx\neq \bsy}}\exp\left( 2\pi i \bsh_{1:d}\cdot (\bsx-\bsy)\right)\prod_{j=d+1}^{s}\int_0^1\int_0^1 \exp\left( 2\pi i h_j\cdot (x_j-y_j)\right)\rd x_j\rd y_j \right] \\
    & = \sum_{\bsh\in \ZZ^s\setminus \{\bszero\}}r_{\alpha,\bsgamma}(\bsh) \cdot \left[\frac{1}{N} +\frac{1}{N^2}\sum_{\substack{\bsx,\bsy\in P_{N,d,\bsz}\\ \bsx\neq \bsy}}\exp\left( 2\pi i \bsh_{1:d}\cdot (\bsx-\bsy)\right)\prod_{j=d+1}^{s}1_{h_j=0} \right] \\
    & = \frac{1}{N}\sum_{\bsh\in \ZZ^s\setminus \{\bszero\}}r_{\alpha,\bsgamma}(\bsh)+\frac{1}{N^2}\sum_{\substack{\bsh\in \ZZ^s\setminus \{\bszero\}\\ \bsh_{d+1:s}=\bszero}}r_{\alpha,\bsgamma}(\bsh)\sum_{\substack{\bsx,\bsy\in P_{N,d,\bsz}\\ \bsx\neq \bsy}}\exp\left( 2\pi i \bsh_{1:d}\cdot (\bsx-\bsy)\right) \\
    & = \frac{1}{N}\sum_{\substack{\bsh\in \ZZ^s\setminus \{\bszero\}\\ \bsh_{d+1:s} \neq\bszero}}r_{\alpha,\bsgamma}(\bsh)+\frac{1}{N^2}\sum_{\substack{\bsh\in \ZZ^s\setminus \{\bszero\}\\ \bsh_{d+1:s}=\bszero}}r_{\alpha,\bsgamma}(\bsh)\sum_{\bsx,\bsy\in P_{N,d,\bsz}}\exp\left( 2\pi i \bsh_{1:d}\cdot (\bsx-\bsy)\right) \\
    & = \frac{1}{N}\sum_{\substack{\bsh\in \ZZ^s\setminus \{\bszero\}\\ \bsh_{d+1:s} \neq\bszero}}r_{\alpha,\bsgamma}(\bsh)+\sum_{\substack{\bsh\in \ZZ^s\setminus \{\bszero\}\\ \bsh_{1:d}\in P^{\perp}_{N,d,\bsz}\\ \bsh_{d+1:s}=\bszero}}r_{\alpha,\bsgamma}(\bsh) \\
    & = \frac{1}{N}\sum_{\substack{\bsh\in \ZZ^s\setminus \{\bszero\}\\ \bsh_{d+1:s} \neq\bszero}}r_{\alpha,\bsgamma}(\bsh)+\sum_{\bsh_{1:d}\in P^{\perp}_{N,d,\bsz}\setminus \{\bszero\}}r_{\alpha,\bsgamma}(\bsh_{1:d},\bszero).
\end{align*}
Thus we are done.
\end{proof}

Now we show the main result of this paper.
\begin{theorem}\label{thm:main}
Let $N,s,d$ be positive integers with $d<s$ and $\alpha>1/2$ be a real number. Assume that there exists $0<\lambda^* < 1$ such that
\begin{align}\label{eq:summable}
\sum_{j=1}^{\infty}\gamma_j^{\lambda^*}<\infty.
\end{align}
Let $\bsz\in \{1,\ldots,N-1\}^d$ be constructed by the (fast) component-by-component algorithm with the quality criterion $(e^{\wor}(H_{d,\alpha,\bsgamma},P_{N,d,\bsgamma}))^2$. Then the mean squared worst-case error of an equally-weighted quadrature rule using a point set $P_{N,s,d,\bsz}$ for the weighted Korobov space $H_{s,\alpha,\bsgamma}$ is bounded above by
\begin{align}
    \EE\left[(e^{\wor}(H_{s,\alpha,\bsgamma},P_{N,s,d,\bsz}))^2\right] & \leq \left[ \frac{1}{\varphi(N)}\left(-1+\exp\left(2C^{\lambda}\zeta(2\alpha \lambda)\zeta(\lambda/\lambda^*) \right)\right)\right]^{1/\lambda} \notag \\
    & \quad + \frac{\exp\left( 2C\zeta(2\alpha)\zeta(1/\lambda^*) \right)}{Nd^{1/\lambda^*-1}}\left[-1 + \exp\left( \frac{2C\zeta(2\alpha)}{1/\lambda^*-1}\right)\right],\label{eq:error_bound}
\end{align}
for any $\lambda\in (\max(\lambda^*,1/(2\alpha)),1]$,
where $C>0$ is a constant depending only on $\bsgamma$, and $\varphi(\cdot)$ and $\zeta(\cdot)$ denote the Euler totient function and the Riemann zeta function, respectively,
\end{theorem}

\begin{remark}\label{rem:totient}
It is shown in \cite[Theorem~15]{RS62} that 
\[ \frac{1}{\varphi(N)} \leq \frac{1}{N}\left( e^c \log\log N+\frac{2.50637}{\log\log N}\right)\]
for any $N\geq 3$ with $c = 0.577\ldots$ being the Euler’s constant.
\end{remark}

\begin{proof}[Proof of Theorem~\ref{thm:main}]
It suffices from Lemma~\ref{lem:ms_worst-case_error} to give upper bounds on the two terms on the right-hand side of \eqref{eq:ms_worst-case_error}, respectively. The second term on the right-hand side of \eqref{eq:ms_worst-case_error} is exactly the squared worst-case error for a rank-1 lattice rule in $H_{d,\alpha,\bsgamma}$ (not in $H_{s,\alpha,\bsgamma}$, of course). By choosing a generating vector $\bsz$ by the standard (fast) component-by-component algorithm, it can be bounded above by
\[ \sum_{\bsh_{1:d}\in P^{\perp}_{N,d,\bsz}\setminus \{\bszero\}}r_{\alpha,\bsgamma}(\bsh_{1:d},\bszero) \leq \left[ \frac{1}{\varphi(N)}\left(-1+\prod_{j=1}^{d}\left[ 1+2\gamma_j^{\lambda}\zeta(2\alpha \lambda)\right] \right)\right]^{1/\lambda},\]
for any $\lambda\in (1/(2\alpha),1]$, see \cite[Theorem~5.12]{DKS13}. Given the summability condition \eqref{eq:summable}, we can assume the existence of $C>0$, which depends only on $\bsgamma$, such that
\begin{align}\label{eq:weight_decay}
\gamma_j\leq Cj^{-1/\lambda^*}
\end{align}
for all $j\in \NN$. By using the inequality $\log(1+x)\leq x$ for any $x>0$, it holds for any $\lambda\in (\max(\lambda^*,1/(2\alpha)),1]$ that
\begin{align*}
    \sum_{\bsh_{1:d}\in P^{\perp}_{N,d,\bsz}\setminus \{\bszero\}}r_{\alpha,\bsgamma}(\bsh_{1:d},\bszero) & \leq \left[ \frac{1}{\varphi(N)}\left(-1+\exp\left(\sum_{j=1}^{d}\log( 1+2\gamma_j^{\lambda}\zeta(2\alpha \lambda))\right) \right)\right]^{1/\lambda} \\
    & \leq \left[ \frac{1}{\varphi(N)}\left(-1+\exp\left(2\zeta(2\alpha \lambda)\sum_{j=1}^{d}\gamma_j^{\lambda} \right)\right)\right]^{1/\lambda} \\
    & \leq \left[ \frac{1}{\varphi(N)}\left(-1+\exp\left(2C^{\lambda}\zeta(2\alpha \lambda)\sum_{j=1}^{\infty}j^{-\lambda/\lambda^*} \right)\right)\right]^{1/\lambda} \\
    & = \left[ \frac{1}{\varphi(N)}\left(-1+\exp\left(2C^{\lambda}\zeta(2\alpha \lambda)\zeta(\lambda/\lambda^*) \right)\right)\right]^{1/\lambda}.
\end{align*}

Let us look at the first term on the right-hand side of \eqref{eq:ms_worst-case_error}. We have
\begin{align*}
    \frac{1}{N}\sum_{\substack{\bsh\in \ZZ^s\setminus \{\bszero\}\\ \bsh_{d+1:s} \neq\bszero}}r_{\alpha,\bsgamma}(\bsh) & = \frac{1}{N}\sum_{\substack{\bsh_{1:d}\in \ZZ^d\\ \bsh_{d+1:s}\in \ZZ^{s-d}\setminus \{\bszero\}}}\prod_{\substack{j=1\\ h_j\neq 0}}^{s}\frac{\gamma_j}{|h_j|^{2\alpha}} \\
    & = \frac{1}{N}\left(\sum_{\bsh_{1:d}\in \ZZ^d}\prod_{\substack{j=1\\ h_j\neq 0}}^{d}\frac{\gamma_j}{|h_j|^{2\alpha}}\right)\left(\sum_{ \bsh_{d+1:s}\in \ZZ^{s-d}\setminus \{\bszero\}}\prod_{\substack{j=d+1\\ h_j\neq 0}}^{s}\frac{\gamma_j}{|h_j|^{2\alpha}}\right) \\
    & = \frac{1}{N}\prod_{j=1}^{d}\left[1+2\gamma_j\zeta(2\alpha) \right]\left(-1+\prod_{j=d+1}^{s}\left[1+2\gamma_j\zeta(2\alpha) \right]\right).
\end{align*}
 Then, by taking account of the decay \eqref{eq:weight_decay} and using the inequality $\log(1+x)\leq x$ again, we have
\begin{align*}
    \prod_{j=1}^{d}\left[1+2\gamma_j\zeta(2\alpha) \right] & = \exp\left( \sum_{j=1}^{d}\log(1+2\gamma_j\zeta(2\alpha) ) \right) \\
    & \leq \exp\left( 2\zeta(2\alpha)\sum_{j=1}^{d}\gamma_j \right) \\
    & \leq \exp\left( 2C\zeta(2\alpha)\sum_{j=1}^{\infty}j^{-1/\lambda^*} \right) \\
    & = \exp\left( 2C\zeta(2\alpha)\zeta(1/\lambda^*) \right),
\end{align*}
and
\begin{align*}
    -1+\prod_{j=d+1}^{s}\left[1+2\gamma_j\zeta(2\alpha) \right] & = -1 + \exp\left( \sum_{j=d+1}^{s}\log(1+2\gamma_j\zeta(2\alpha))\right) \\
    & \leq -1 + \exp\left( 2\zeta(2\alpha)\sum_{j=d+1}^{s}\gamma_j\right) \\
    & \leq -1 + \exp\left( 2C\zeta(2\alpha)\sum_{j=d+1}^{\infty}j^{-1/\lambda^*}\right) \\
    & \leq -1 + \exp\left( 2C\zeta(2\alpha)\int_{d}^{\infty}x^{-1/\lambda^*}\rd x\right) \\
    & = -1 + \exp\left( \frac{2C\zeta(2\alpha)}{1/\lambda^*-1}d^{-1/\lambda^*+1}\right) \\
    & \leq \left[-1 + \exp\left( \frac{2C\zeta(2\alpha)}{1/\lambda^*-1}\right)\right]d^{-1/\lambda^*+1}.
\end{align*}
Here the last inequality is obtained as follows: given that the function $f(x)= -1+\exp(x)$ is convex, Jensen's inequality leads to
\begin{align*}
    -1 + \exp\left( \frac{2C\zeta(2\alpha)}{1/\lambda^*-1}d^{-1/\lambda^*+1}\right) & = f\left( \frac{2C\zeta(2\alpha)}{1/\lambda^*-1}d^{-1/\lambda^*+1}\right) \\
    & \leq (1-d^{-1/\lambda^*+1})f(0)+d^{-1/\lambda^*+1}f\left( \frac{2C\zeta(2\alpha)}{1/\lambda^*-1}\right) \\
    & = \left[-1 + \exp\left( \frac{2C\zeta(2\alpha)}{1/\lambda^*-1}\right)\right]d^{-1/\lambda^*+1}.
\end{align*}
Altogether the first term on the right-hand side of \eqref{eq:ms_worst-case_error} is bounded by
\[ \frac{1}{N}\sum_{\substack{\bsh\in \ZZ^s\setminus \{\bszero\}\\ \bsh_{d+1:s} \neq\bszero}}r_{\alpha,\bsgamma}(\bsh) \leq \frac{\exp\left( 2C\zeta(2\alpha)\zeta(1/\lambda^*) \right)}{Nd^{1/\lambda^*-1}}\left[-1 + \exp\left( \frac{2C\zeta(2\alpha)}{1/\lambda^*-1}\right)\right]. \]
Since both the terms on the right-hand side of \eqref{eq:ms_worst-case_error} are now bounded above as desired, we are done.
\end{proof}

It is important that our obtained upper bound on the mean squared worst-case error is independent of the problem dimension $s$. Assuming $s\gg N,d$ or $s\to \infty$, let us discuss a choice of $d$.
\begin{enumerate}
    \item If $\lambda^*\leq 1/(2\alpha)$, the bound \eqref{eq:error_bound} holds for any $\lambda\in (1/(2\alpha),1]$. Since it is known that we cannot achieve a convergence rate better than $O(N^{-\alpha})$ for the worst-case error, see \cite[Theorem~4.22]{LP14} for a lower bound on the worst-case error, the order of the first term of \eqref{eq:error_bound} is almost optimal, as evident from Remark~\ref{rem:totient}. Here we recall that, although we consider the average of the squared worst-case error in this note, the square of a lower bound on the deterministic worst-case error still applies; otherwise a contradiction occurs. Thus, in order for the second term not to be dominant, it suffices to choose
    \[ d \propto N^{(2\alpha-1)/(1/\lambda^*-1)}<N.\]
    For such a choice, the required construction cost for the first coordinates is of order
    \[ dN\log N \propto N^{1+(2\alpha-1)/(1/\lambda^*-1)}\log N.\]
    A substantial cost saving can be expected for smaller $\lambda^*$, as compared to the standard algorithm of order $sN\log N$.
    \item If $1/(2\alpha)<\lambda^*\leq 1$, on the other hand, the bound \eqref{eq:error_bound} holds only for the range $\lambda\in (\lambda^*,1]$, which does not lead to an optimal order error bound. Nevertheless, in order for the second term not to be dominant, it suffices to choose
    \[ d \propto N\]
    regardless of the value of $\lambda^*$, resulting in the necessary construction cost of order
    \[ dN\log N \propto N^2\log N.\]
    This might be useful when $s$ is extremely large.
\end{enumerate}

\subsection{Comparison with the reduced CBC algorithm}

Let us consider rank-1 lattice rules constructed the reduced CBC algorithm introduced in \cite{DKLP15}, in which $\alpha$ should be replaced by $2\alpha$ to be consistent with this note. The worst-case error is bounded independently of the dimension $s$ if there exist $0<\lambda^*<1$ and $0\leq w_1\leq w_2\leq \cdots$ with $w_j\in \NN \cup \{0\}$ such that
\begin{align}\label{eq2:summable} \sum_{j=1}^{\infty}\gamma_j^{\lambda^*}b^{w_j}<\infty
\end{align}
holds for a fixed prime $b$. The squared worst-case error decays with order $N^{-1/\lambda}$ for any $\lambda\in (\max(\lambda^*,1/(2\alpha)),1]$. Note that such results for the reduced CBC algorithm are deterministic on the contrary to our concatenated rule. Also, because of the shift-invariance of the kernel $K_{s,\alpha,\bsgamma}$, the same result holds for the mean squared worst-case error of a randomly shifted rank-1 lattice rule. When the number of points is given by $N=b^m$ for some $m\in \NN$, the required construction cost for the reduced CBC algorithm, improved in \cite{EKN20}, is of order
\begin{align}\label{eq:reduced_CBC_cost}
    \sum_{j=1}^{\min(s,s^*)}(m-w_j)b^{m-w_j},
\end{align} 
where $s^*$ is defined to be the largest integer such that $w_{s^*}< m$. In the following, assuming that the weights are given by $\gamma_j=j^{-2c\alpha}$ for some $c>1/(2\alpha)$, we show that this construction cost is in general not comparable to that for our concatenated rule.

\paragraph{Slowly decaying weights} First let us consider the case $1/(2\alpha)<c<1$. The condition \eqref{eq:summable} holds for $\lambda^*=1/(2c\alpha)+\epsilon$ with arbitrarily small $\epsilon>0$. As discussed above, the necessary construction cost for our concatenated rule is of $O(N^2\log N)$. Regarding the reduced CBC algorithm, the condition \eqref{eq2:summable} holds trivially for $\lambda^*=1/(2c\alpha)+\epsilon$ if $w_1= w_2= \cdots=0$. Then the reduced CBC algorithm coincides with the standard CBC algorithm, requiring the necessary construction cost of $O(sN\log N)$. By setting $w_j=\lfloor \beta \log_b j\rfloor$ for some $0<\beta<2c\alpha-1$, we can set $s^*\leq N^{1/\beta}$ and it follows from \eqref{eq:reduced_CBC_cost} that the order of construction cost becomes at most
\begin{align}\label{eq:reduced_CBC_cost2}
    \sum_{j=1}^{\min(s,s^*)}(m-w_j)b^{m-w_j}\leq mb^m\sum_{j=1}^{N^{1/\beta}}\frac{b}{j^\beta}\leq \begin{cases} b\zeta(\beta)N\log_b N & \text{for $\beta>1$,} \\ (b\log b) N(\log_b Ne)^2 & \text{for $\beta=1$,}\\ bN^{1/\beta}\log_b N & \text{otherwise.}\end{cases}
\end{align}   
Thus the dividing case in terms of construction cost is $\beta=1/2$. However, for any $\beta>0$, the condition \eqref{eq2:summable} holds for $\lambda^*=(\beta+1)/(2c\alpha)+\epsilon$, which deteriorates the decay of the worst-case error.

\paragraph{Fast decaying weights} Let us consider the case $c>1$ next. The condition \eqref{eq:summable} holds for $\lambda^*=1/(2c\alpha-\epsilon)$ with arbitrarily small $\epsilon>0$. The necessary construction cost for our concatenated rule is of order $N^a\log N$ with \[ a=1+\frac{2\alpha-1}{2c\alpha-\epsilon-1}.\] 
Regarding the reduced CBC algorithm, by setting $w_j=\lfloor \beta \log_b j\rfloor$ for some $0<\beta<2c\alpha-1$, the condition \eqref{eq2:summable} holds for $\lambda^*=(\beta+1)/(2c\alpha)+\epsilon$. Looking at \eqref{eq:reduced_CBC_cost2}, the dividing case in terms of construction cost is
\[ \beta=\frac{1}{a}=\left(1+\frac{2\alpha-1}{2c\alpha-\epsilon-1}\right)^{-1}<1.\]
However, in this case, the reduced CBC algorithm can achieve almost the optimal order of the worst-case error as long as $\beta<c(1-2\alpha\epsilon)-1$. Thus, if $c$ is large enough, we can set $\beta\geq 1$, providing a smaller construction cost than our concatenated rule.

\subsection{Comparison with truncation algorithm}
Finally we compare the result for our concatenated rule with that for a truncation algorithm inspired by \cite{KPW16}. Let us consider the point set 
\[ P_{N,s,d,\bsz,\bsc}^{\tr} = \left\{ \left( \left\{ \frac{nz_1}{N}\right\},\ldots,\left\{ \frac{nz_d}{N}\right\},c_{d+1},\ldots,c_s\right) \mid 0\leq n\leq N-1\right\},\]
for an anchor vector $\bsc=(c_{d+1},\ldots,c_s)\in [0,1)^{s-d}$, which can be fixed or chosen independently uniformly from $[0,1)^{s-d}$. We obtain the following worst-case error bound for the truncation algorithm. Although a similar result can be found, for instance, in \cite[Theorem~13]{KPW16}, we give a proof for the sake of completeness.

\begin{theorem}\label{thm2:main}
Let $N,s,d$ be positive integers with $d<s$ and $\alpha>1/2$ be a real number. Assume that there exists $0<\lambda^* < 1$ such that \eqref{eq:summable} holds. Let $\bsz\in \{1,\ldots,N-1\}^d$ be constructed by the (fast) component-by-component algorithm with the quality criterion $(e^{\wor}(H_{d,\alpha,\bsgamma},P_{N,d,\bsgamma}))^2$. Then, for any $\bsc\in [0,1)^{s-d}$, the squared worst-case error of an equally-weighted quadrature rule using a point set $P_{N,s,d,\bsz,\bsc}^{\tr}$ for the weighted Korobov space $H_{s,\alpha,\bsgamma}$ is bounded above by
\begin{align}
    & (e^{\wor}(H_{s,\alpha,\bsgamma},P_{N,s,d,\bsz,\bsc}^{\tr}))^2 \notag \\
    & \leq \left[ \frac{1}{\varphi(N)}\left(-1+\exp\left(2C^{\lambda}\zeta(2\alpha \lambda)\zeta(\lambda/\lambda^*) \right)\right)\right]^{1/\lambda}  \exp\left( 2C\zeta(2\alpha)\zeta(1/\lambda^*) \right) \notag \\
    & \quad + \frac{1}{d^{1/\lambda^*-1}}\left[-1 + \exp\left( \frac{2C\zeta(2\alpha)}{1/\lambda^*-1}\right)\right],\label{eq2:error_bound}
\end{align}
for any $\lambda\in (\max(\lambda^*,1/(2\alpha)),1]$, where $C>0$ is a constant depending only on $\bsgamma$.
\end{theorem}
\begin{proof}
Because of the shift-invariance of the kernel $K_{s,\alpha,\bsgamma}$, it suffices to prove the statement for the case $\bsc=\bszero$.
Following an argument similar to the proof of Lemma~\ref{lem:ms_worst-case_error}, we have
\begin{align*}
    (e^{\wor}(H_{s,\alpha,\bsgamma},P_{N,s,d,\bsz,\bszero}^{\tr}))^2 & = \frac{1}{N^2}\sum_{\bsx,\bsy\in P_{N,s,d,\bsz,\bszero}^{\tr}}\sum_{\bsh\in \ZZ^s\setminus \{\bszero\}}r_{\alpha,\bsgamma}(\bsh)\exp\left( 2\pi i \bsh\cdot (\bsx-\bsy)\right) \\
    & = \sum_{\bsh\in \ZZ^s\setminus \{\bszero\}}r_{\alpha,\bsgamma}(\bsh)\frac{1}{N^2}\sum_{\bsx,\bsy\in P_{N,d,\bsz}}\exp\left( 2\pi i \bsh_{1:d}\cdot (\bsx-\bsy)\right) \\
    & = \sum_{\substack{\bsh\in \ZZ^s\setminus \{\bszero\}\\ \bsh_{1:d}\in P^{\perp}_{N,d,\bsz}}}r_{\alpha,\bsgamma}(\bsh) \\
    & = \sum_{\substack{\bsh_{1:d}\in P^{\perp}_{N,d,\bsz}\setminus \{\bszero\}\\ \bsh_{d+1:s}\in \ZZ^{s-d}}}r_{\alpha,\bsgamma}(\bsh) + \sum_{\bsh_{d+1:s}\in \ZZ^{s-d}\setminus \{\bszero\}}r_{\alpha,\bsgamma}(\bszero,\bsh_{d+1:s}).
\end{align*}
As in the proof of Theorem~\ref{thm:main}, with $C$ being a positive constant depending only on $\bsgamma$, the first and second terms are bounded above by
\begin{align*}
    \sum_{\substack{\bsh_{1:d}\in P^{\perp}_{N,d,\bsz}\setminus \{\bszero\}\\ \bsh_{d+1:s}\in \ZZ^{s-d}}}r_{\alpha,\bsgamma}(\bsh) & = \sum_{\bsh_{1:d}\in P^{\perp}_{N,d,\bsz}\setminus \{\bszero\}}r_{\alpha,\bsgamma}(\bsh_{1:d},\bszero)\sum_{\bsh_{d+1:s}\in \ZZ^{s-d}}r_{\alpha,\bsgamma}(\bszero,\bsh_{d+1:s}) \\
    & \leq \left[ \frac{1}{\varphi(N)}\left(-1+\exp\left(2C^{\lambda}\zeta(2\alpha \lambda)\zeta(\lambda/\lambda^*) \right)\right)\right]^{1/\lambda}  \prod_{j=d+1}^{s}\left[1+2\gamma_j\zeta(2\alpha) \right] \\
    & \leq \left[ \frac{1}{\varphi(N)}\left(-1+\exp\left(2C^{\lambda}\zeta(2\alpha \lambda)\zeta(\lambda/\lambda^*) \right)\right)\right]^{1/\lambda}  \exp\left( 2C\zeta(2\alpha)\zeta(1/\lambda^*) \right),
\end{align*}
for any $\lambda\in (\max(\lambda^*,1/(2\alpha)),1]$ and
\begin{align*}
    \sum_{\bsh_{d+1:s}\in \ZZ^{s-d}\setminus \{\bszero\}}r_{\alpha,\bsgamma}(\bszero,\bsh_{d+1:s}) & = -1+\prod_{j=d+1}^{s}\left[1+2\gamma_j\zeta(2\alpha) \right] \\
    & \leq \left[-1 + \exp\left( \frac{2C\zeta(2\alpha)}{1/\lambda^*-1}\right)\right]d^{-1/\lambda^*+1},
\end{align*}
respectively, which completes the proof.
\end{proof}

In contrast to the bound \eqref{eq:error_bound} for our concatenated rule, the one \eqref{eq2:error_bound} we show here is a deterministic worst-case error bound, which is a clear advantage of the truncation algorithm. On the other hand, in order for the second term of \eqref{eq2:error_bound} not to be dominant, the truncation dimension $d$ should be larger than that for our concatenated rule. This is simply because the $N$ factor is missing in the denominator of the second term as compared to \eqref{eq:error_bound}. In fact, a sufficient choice for the truncation algorithm is given by
\[ d\propto \begin{cases} N^{2\alpha/(1/\lambda^*-1)} & \text{if $0<\lambda^*\leq 1/(2\alpha)$,} \\ N^{1/(1-\lambda^*)} & \text{if $1/(2\alpha)<\lambda^*<1$,} \end{cases}
\]
which can make a significant difference of the necessary construction cost for the first $d$ coordinates. This is a point where randomization helps for our high-dimensional integration problems.

\section*{Acknowledgements}
The author would like to thank Kei Ishikawa (ETH Z\"{u}rich) for bringing the problem to his attention. He is very grateful to Michael Gnewuch (Osnabr\"{u}ck) for pointing out some relevant literature and also an error contained in the first version of this manuscript. He also appreciates the comments and suggestions from the anonymous reviewers. The work of T.G.\ is supported by JSPS KAKENHI Grant Number 20K03744.

\bibliographystyle{plain}
\bibliography{ref.bib}

\end{document}